\documentclass[11pt,oneside]{amsart}
\usepackage{amssymb, amscd, amsmath, amsthm, latexsym, hyperref,xcolor}
 \usepackage{pb-diagram, fancyhdr, graphicx, psfrag}
 \usepackage[text={6.3in,8.5in},centering,letterpaper,dvips]{geometry}

\newcommand{\compactlist}{\begin{list}{$\bullet$}{\setlength{\leftmargin}{1em}}}

\newcommand{\fig}[2] { \includegraphics[scale=#1]{#2} }

\def\co{\colon\thinspace}
\def\cs{\mathop{\#}}

\newcommand{\spinc}{\ifmmode{{\mathfrak s}}\else{${\mathfrak s}$\ }\fi}
\newcommand{\spinct}{\ifmmode{{\mathfrak t}}\else{${\mathfrak t}$\ }\fi}
\newcommand{\spincw}{\ifmmode{{\mathfrak w}}\else{${\mathfrak w}$\ }\fi}

\def\sp{\mathop{\mathrm{sp}}}

\newtheorem{theorem}{Theorem}[section]

\newtheorem*{theorem*}{Theorem}
\newtheorem{lemma}[theorem]{Lemma}
\newtheorem{corollary}[theorem]{Corollary}

\theoremstyle{definition}
\newtheorem{example}[theorem]{Example}
\newtheorem{definition}[theorem]{Definition}

\theoremstyle{remark}

\numberwithin{equation}{section}

\begin{document}


\title[Splitting numbers of links and the four-genus]{Splitting numbers of links and  the four-genus}
\author{Charles Livingston}
 \address{Charles Livingston: Department of Mathematics, Indiana University, Bloomington, IN 47405 }
\email{livingst@indiana.edu}

\thanks{The author was supported by a Simons Foundation grant and by NSF-DMS-1505586.}

 
 \maketitle
 
 \begin{abstract}   The splitting number of a link is the minimum number of crossing changes between distinct components that is required to convert the link into a split link.  We provide a bound on the splitting number in terms of the four-genus of related knots.  
 \end{abstract}
 
 \section{Introduction}
     
 A  link  $L  \subset S^3$  has {\it splitting number} 
$\sp (L) = n$ if $n$ is the least nonnegative integer for which some choice of  $n$ crossing changes between distinct components  results in a  totally split link.  The study of splitting numbers and closely related invariants includes~\cite{adams, batson-seed, bfp, bg, cfp, kohn, shimizu}.  (In  \cite{adams, kohn,   shimizu}, the term splitting number permits self-crossing changes.)  Here we will investigate the splitting number from the perspective of the four-genus of knots, an  approach that is closely related to the use of concordance to study the splitting number in~\cite{bfp, bg} and earlier work considering concordances to split links~\cite{kawauchi}.  Recent work by Jeong~\cite{jeong}  develops a new infinite family of   invariants that bound the splitting number, based on Khovanov homology.  We will be working in the category of smooth oriented links, but notice that the splitting number is independent of the choice of orientation.  
   
   To state our results, we remind the reader of the notion of  a band connected sum of a  link $L$.  A {\it band} $b$ is an embedding  $b\co [0,1] \times [0,1] \to S^3$ such that $\text{Image}(b) \cap L = b([0,1] \times \{0, 1\}) $.   The orientation of the band must be consistent with the orientation of the link.  The link $L_b$ is $L  \setminus (  b([0,1]\times \{0, 1\}) ) \cup (b(\{0,1\} \times [0,1]$).  Similarly, for a link $L$ of $k$ components, we can consider a set of $k-1$ disjoint bands $\beta = \{b_1, \ldots , b_{k-1}\}$ and use these to construct a link $L_{\beta}$; we will always work in the setting that $\beta$ has the property that $L_{\beta}$ is connected.  We will call such a set of bands a {\it minimal connecting set of bands}.
   
   For a knot $K$, we denote the mirror image of $K$ with string orientation reversed by $\overline{K}$.

   \begin{theorem}\label{thm:main} Let $L =L_1 \cup \ldots \cup L_k $ be an oriented $k$--component link with linking numbers $lk(L_i,L_j) =l_{i,j}$ for $i \ne j$, and let $\beta$ be a minimal connecting set of  bands.  Let $N$ be the total linking number: $N =  \left|\sum_{ i<j} l_{i,j}\right|$. Then $$sp(L) \ge 2g_4(L_\beta  \#_{i=1}^{k} \overline{L}_i) - N.$$
 \end{theorem}
   
   As a simple corollary, we have:
   
   \begin{corollary}\label{cor:main} If $L $ is a $k$--component link with unknotted components and with all linking numbers 0, then for any minimal set of connecting bands, $ sp(L) \ge 2g_4(L_\beta)$.   \end{corollary}
   
   \begin{example}  The simplest nontrivial link, the Hopf link, illustrates the role of the choice of $\beta$.  One band connected sum yields the unknot, with four-genus 0 and another yields the trefoil, with four-genus 1; from this,  Theorem~\ref{thm:main} implies the obvious, that the splitting number is 1.
   \end{example}
   
   \begin{example} The first two basic examples of non-split links which are algebraic unlinked are the Whitehead link and the Borromean link.  Most tools for studying splittings handle these examples, as does Corollary~\ref{cor:main}. For both links, band moves yield the trefoil knot, of four-genus 1, showing the splitting number is at least 2.  Splittings with exactly two crossing changes are easily constructed.
   
    Here is a generalization of an example in~\cite{bg},  studied in more depth in~\cite{ccz}.   Consider the two-bridge  link illustrated in Figure~\ref{fig:examples}, with all of $m, n$ and $l$ nonnegative.  The numbers in the boxes represent full twists.  Without loss of generality, we can assume $m \ge n$.  The linking number is $ m-n $. The illustrated band leads to a  knot $L_b$ whose signature is easily computed to be $-2m$, so $g_4(L_b) \ge  m$.  In fact, $L_b$ is the connected sum of the torus knot $T_{2,2m+1}$  and a genus 1 knot of signature 0.  Thus, by Theorem~\ref{thm:main},  $sp(L) \ge   2m - (m-n) = n +m$.  The link can evidently  be split with $n+m$ crossings changes, so $sp(L) = n+m$.
 \begin{figure}[h]
\fig{.5}{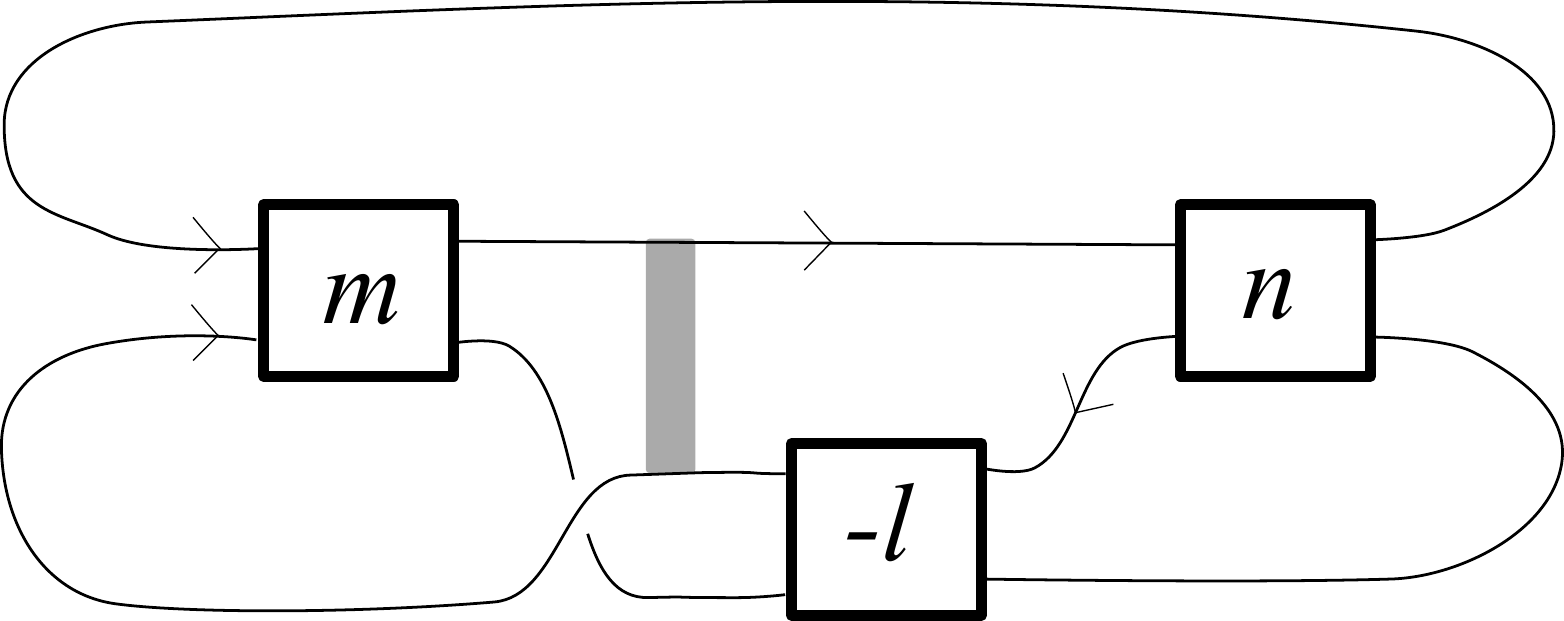}  \caption{A family of two-bridge links.}\label{fig:examples}
\end{figure}
 
  The cases of $(m,n,l) = (1,2,1)$ and $(m,n,l) = (2,3,1)$ are the links   $L_{9a30}$ and $L_{11a372}$. The splitting numbers of these were determined in~\cite{cfp}, with $L_{9a30}$ serving as a basic example and $L_{11a372}$ as an example of a case which could not be resolved in~\cite{batson-seed}.
  \end{example}
   Theorem~\ref{thm:main} provides a surprisingly easy and effective tool in determining splitting numbers, but it is not difficult to find examples for which it is   weaker than previously developed methods.  One reason is that the bound given in Theorem~\ref{thm:main} is in fact a bound on the {\it concordance splitting number}, $csp(L)$, which is implicitly studied in~\cite{bg}.  This invariant is discussed in Section~\ref{sec:csp}.  The next family of examples presents the distinction between the two invariants.
   
   \begin{example}
   Figure~\ref{fig:bing}\label{exp:bing} illustrates a link $L_K$, the Bing double of a knot $K$. The presence of an incompressible torus in its complement shows that if $K$ is nontrivial, then $sp(L_K) = 2$.   If $K$ is slice, then $L_K$ is concordant to the unlink, so $csp(L) = 0$ and the splitting number cannot be detected by Theorem~\ref{thm:main}.  
   
   The indicated band move on the Bing double  produces the untwisted Whitehead double, $Wh(K)$.  Thus, by Corollary~\ref{cor:main}, if $K$ is such that  $Wh(K)$ is not slice, then $L_K$ is not concordant to a split link.  As an example, letting $K = Wh(T_{2,3})$ yields an example of a link $L_K$ which is topologically but not smoothly concordant to a split link.  Presumably, algebraic invariants would not detect the splitting number in this case.  
   
   Alternative approaches to showing the concordance splitting number of $L_K$ is 2 (for specific choices of $K$)  can be based on showing that the Bing double is not strongly slice, which was done, for instance, in~\cite{clr, cimasoni}.
\begin{figure}[h]
\fig{.45}{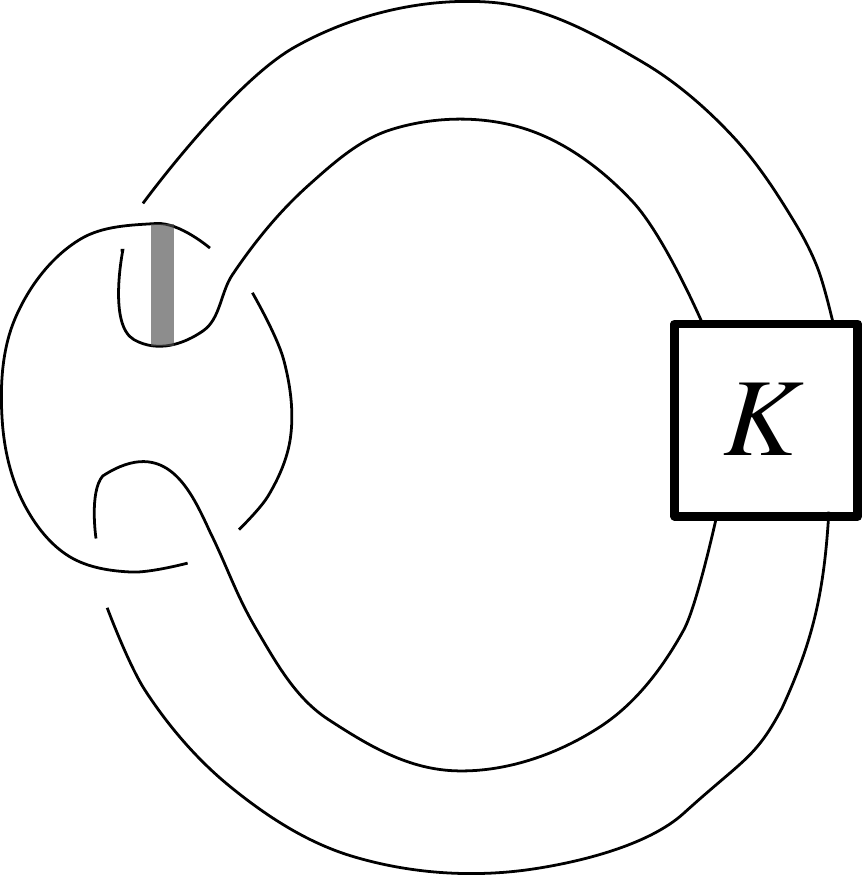}  \caption{The Bing double, $L_K$.}\label{fig:bing}
\end{figure}

\end{example}

   \section{Proof of Theorem~\ref{thm:main}}
   
   \subsection{The trace of the isotopy}
A set of crossing changes in $L$ corresponds to an isotopy of $L$ with double points; the trace of this isotopy in $S^3\times [0,1]$  is an {\it  immersed concordance}.  To be specific, an immersed concordance between  $k$--component links  $L^0$ and $L^1 $ is a smooth immersion $$F\co  S^1 \times [0,1] \times \{1,\ldots, k\} \to S^3 \times [0,1]$$ such that  $$F(S^1 \times i \times j) = L^i_j \subset S^3 \times i$$ for $i = 0,1$ and $j= 1,\ldots , k$.   Singular points are requred to be isolated transverse double points.

 In the setting of Theorem~\ref{thm:main}, $L^0 = L$, $L^1 = L_1 \cup \cdots \cup L_k$  (the split link with components the knots $L_i$) and the immersed concordance consists of a set of $k$ {\it embedded} concordances intersecting transversely in double points. These embedded concordances are called {\it the components} of the immersed concordance, although they need not be disjoint.

Projection of $S^3 \times [0,1]$ onto $[0,1]$ defines a height function.  In the current situation, there are no critical points for the height function on the concordance; each component is an embedded product concordance.  

The isotopy of $L^0$ can a be extended to the bands by isotoping them so  that they do not interfere with the crossings.  Thus, we construct an immersed concordance from $L^0_\beta$ to $L^1_{\beta'}$ for some set of bands $\beta'$.

 \subsection{Forming the connected sum with the $\overline{L}_i$.}
We now form the connected sum of $L_\beta $ with $ \cs_i \overline{L}_i$.  Do this by forming the connected sum of each $L_i$ with the corresponding $\overline{L}_i$, so that the $\overline{L}_i$ is in a small ball far from the basepoints of any $b \in \beta$.   It is now clear that  we can modify the immersed concordance to form an  immersed concordance  from  $L_\beta \cs_i \overline{L}_i$ to $(\cs_i (L_i \cs\overline{L}_i))_{\beta'}$.

\subsection{Forming an immersed slice disk}\label{sec:slice}
Observe that the knot $(\cs_i (L_i \cs \overline{L}_i))_{\beta'}$ is slice.  A set of $k-1$ band moves (dual to the bands of $\beta'$)  yields the link $(L_1\cs  \overline{L}_1)  \cup \cdots \cup (L_k\cs \overline{L}_k) $.  Since the components are split and each is slice, we see that the original knot is slice.  

Since  the knot $L^0_\beta \cs_i \overline{L}_i$ bounds a singular concordance to a slice knot, it bounds a singular slice disk in $B^4$ with corresponding singular points. 

\subsection{Counting and resolving the double points}
Let $\mathcal P$ be the set of pairs $(i,j), i < j$, such that the linking number $l_{i,j} \ge 0$ and let $\mathcal N$ be the set of pairs  $(i,j), i < j$, such that the linking number $l_{i,j} < 0$.  

For each pair $(i,j) \in \mathcal P$, the number of positive crossing  points  between the $i$ and $j$ components  during the splitting is $|l_{i,j} | + m_{i,j}$   for some $m_{i,j} \ge 0$; the number of negative crossing changes is $m_{i,j}$.

Similarly, for each pair $(i,j) \in \mathcal N$, the number of   negative crossing points  between the $i$ and $j$ components   during the splitting is    $|l_{i,j}| + m_{i,j}$ for some $m_{i,j} \ge 0$, and the number of positive crossing changes is $m_{i,j}$.

It follows from this count that the total number of positive double points is 
$$A = \sum_{(i,j)\in \mathcal P}| l_{i,j} |+\sum_{(i,j)\in \mathcal P \cup \mathcal N} m_{i,j}.$$ 
The number of negative double points is $$B= \sum_{(i,j)\in \mathcal N}| l_{i,j} |+\sum_{(i,j)\in \mathcal P \cup \mathcal N}m_{i,j}.$$ 

\subsection{Building an embedded  surface in the four-ball bounded by $L_\beta \cs_i \overline{L}_i$}  We now assume that the initial sequence of crossing changes was a minimal splitting sequence.
We build an embedded surface bounded by $L_\beta \cs_i \overline{L}_i$ by tubing together pairs of canceling double points and then resolving the remaining double points individually.  The resulting surface has genus $$g = \max(A,B) = (|A  + B| + |A - B|)/2.$$  Using the formulas for $A$ and $B$, this becomes 
$$2g =    ( \sum_{(i,j)\in \mathcal P \cup \mathcal N} |l_{i,j} | + 2\sum_{(i,j)\in \mathcal P \cup \mathcal N}m_{i,j})  + 
\big|(\sum_{(i,j) \in \mathcal P} |l_{i,j}| - \sum_{(i,j) \in \mathcal N} |l_{i,j}|
)\big|.
$$
The expression in the first set of parenthesis equals the splitting number; the second term (the absolute value of the difference of sums) is simply the absolute value of the  sum of the linking numbers, called $N$ in the statement of Theorem~\ref{thm:main}.  Thus,
$$2g = sp(L) + N,$$
and so, as desired, $$sp(L) = 2g -N\ge 2g_4(L_\beta \#_{i=1}^{k}  \overline{L}_i) - N.$$.

\section{Concordance splitting}\label{sec:csp}
The lower bound on the splitting number given in Theorem~\ref{thm:main} is in fact a  bound on the {\it concordance splitting number}.

\begin{definition} A link $L$ has concordance splitting number $csp(L) = n$ if $n$ is the least nonnegative integer such that  there is an immersed concordance form $L$ to a split link having $n$ double points and each component is embedded.
 \end{definition}
 
 \begin{example}  Example~\ref{exp:bing} demonstrates that for some links $csp(L) < sp(L)$.
 \end{example}

Notice that in the definition, the concordance need not be to the link $L_1 \cup \cdots \cup L_k$.  However, we have the following.

\begin{lemma}\label{lemma:fixend} If $csp(L) = n$, then there is an immersed concordance, with $n$ double points and each component embedded, from $L$ to  $L_1 \cup \cdots \cup L_k$.
\end{lemma}
\begin{proof}The end of the immersed concordance is a link $L'_1 \cup \cdots \cup L'_k$.  Since the components of the immersed concordance are embedded, each $L'_i$ is concordant to $L_i$.  Thus, the immersed concordance can be extended using these individual concordances so that the ending link is $L_1 \cup \cdots \cup L_k$.
\end{proof}

We have the following analog of Theorem~\ref{thm:main}

\begin{theorem}Let $L =L_1 \cup \ldots \cup L_k $ be an oriented $k$--component link with linking numbers $lk(L_i,L_j) =l_{i,j}$ for $i \ne j$, and let $\beta$ be a set of $k-1$ bands for which $L_\beta$  is connected.  Let $N$ be the total linking number: $N =  \left|\sum_{ i<j} l_{i,j}\right|$. Then $$csp(L) \ge 2g_4(L_\beta  \#_{i=1}^{k} \overline{L}_i) - N.$$
\end{theorem}

\begin{proof} Much of the proof proceeds as before, but there is one significant difficulty.  The presence of possible maximum points in the concordance prevents one from converting the concordance of the link into a concordance of its band connected sum.  The bands might interfere with capping off unknotted components that arise from index two critical points.  Here is how the proof is adjusted.

The concordance can be modified so that all critical points of index 2 occur at height $3/4$ and all other critical points and double points occur below the height of $1/4$.  At level $1/2$ we have the link $L' \cup U_1 \cup \cdots \cup U_r$, where the $U_r$ form an unlink split from $L'$ (each component of which is capped off at level $3/4$).  

The index 0 and index 1 critical points do not interfere with the constructions used earlier, and from this one finds that there is a genus 0 immersed corbordism from  $L_\beta \cs_i \overline{L}_i$ to $(\cs_i (L_i \cs\overline{L}_i))_{\beta'}\cup U_1 \cup \cdots \cup U_r$.  Notice that the bands in $\beta'$ might link the $U_i$, and this is the point of difficulty.  However, we can use this cobordism to construct an immersed slice disk: perform the  band moves dual to the $\beta'$ to build a  split link with all components slice knots (some are the $L_i \cs \overline{L}_i$ and some are the $U_i$);  these can be capped off to form the immersed slice disk.

The rest of the proof is identical to that of Theorem~\ref{thm:main}.

\end{proof}

\vskip.2in
   
 \noindent{\it Acknowledgements}  The author thanks Maciej Borodzik, Kathryn Bryant, David Cimasoni, Anthony Conway, and Stefan Friedl for their useful feedback on a preliminary draft of this paper.



\begin{thebibliography}{AAA}
\bibitem{adams} C.~Adams, {\em Splitting versus unlinking}, J.~Knot Theory Ramifications {\bf 5} (1996),  295--299.

\bibitem{batson-seed}   J.~Batson and C.~Seed, {\em  A link splitting spectral sequence in Khovanov homology}, \url{arXiv:1303.6240}.

\bibitem{bfp} M.~Borodzik, S.~Friedl, and M.~Powell, {\em Blanchfield forms and Gordian distance},  Journal of the Mathematical Society of Japan. to appear \url{arXiv:1409.8421}.
\bibitem{bg}   M.~Borodzik and E.~Gorsky, {\em  Immersed concordances of links and Heegaard Floer homology}, \url{arXiv:1601.07507}.


\bibitem{cfp}	J.~C.~Cha, S.~Friedl, and M.~Powell, {\em Splitting numbers of links},
\url{arXiv:1308.5638}. 

\bibitem{clr}	J.~C.~Cha, C.~Livingston, and D.~Ruberman, { \em Algebraic and Heegaard-Floer invariants of knots with slice Bing doubles,} Math. Proc. Cambridge Philos. Soc. 144 (2008),  403--410. 

\bibitem{cimasoni} D.~Cimasoni, {\em Slicing Bing doubles},  Alg. Geom. Topol. 6 (2006) 2395--2415.

\bibitem{ccz} D.~Cimasoni, A.~Conway, and K.~Zacharova, {\em  Splitting numbers and signatures,}  \url{arXiv:1601.07871}.

\bibitem{hom}  J.~Hom, {\em Bordered Heegaard Floer homology and the tau-invariant of cable knots}, J.~Topol. {\bf 7} (2014)  287--326.


\bibitem{jeong}  G.~Jeong, {\em A family of link concordance invariants from perturbed $sl(n)$ homology}, \url{arXiv:1608.05781}.

\bibitem{kawauchi} A.~Kawauchi, {\em 
On links not cobordant to split links}, 
Topology 19 (1980),  321--334. 


  \bibitem{kohn} P.~Kohn, {\em Two-bridge links with unlinking number one}, Proc. Amer. Math. Soc.
{\bf 113} (1991),
1135--1147. 


\bibitem{os} P.~Ozsv\'ath and Z.~Szab\'o, {\it  Knot Floer homology and the four-ball genus},  Geom.~Topol. {\bf 7} (2003) 615--639.

\bibitem{shimizu} A.~Shimizu, {\em The complete splitting number of a lassoed link}, 
Topology Appl. {\bf 159} (2012),  959--965. 


 \end{thebibliography}
\end{document}